\documentclass{amsart}
\usepackage{latexsym,amssymb,amsthm,amsmath}

\usepackage{amssymb}
\usepackage{amsmath}

\theoremstyle{plain}
\newtheorem{theorem}{Theorem}[section]
\newtheorem*{Theorem B}{Theorem B}
\newtheorem*{Theorem A}{Theorem A}
\newtheorem{lemma}{Lemma}[section]
\newtheorem{proposition}{Proposition}[section]
\newtheorem{corollary}{Corollary}[section]

\newtheorem{example}{Example}[section]
\numberwithin{equation}{section}

\theoremstyle{remark}
\newtheorem{remark}{Remark}[section]
\newtheorem{acknowledgements}{Acknowledgement}
\title[Geometry of contact skew CR-warped product submanifolds]{Geometry of contact skew CR-warped product submanifolds of Sasakian manifolds}

\author[S. Uddin]{Siraj Uddin}
\address{S. Uddin: Department of Mathematics, Faculty of Science, King Abdulaziz University, 21589 Jeddah, Saudi Arabia}
\email{siraj.ch@gmail.com}
\author[F.R. Al-Solamy]{Falleh R. Al-Solamy}
\address{F. R. Al-Solamy: Department of Mathematics, Faculty of Science, King Abdulaziz University, 21589 Jeddah, Saudi Arabia}
\email{falleh@hotmail.com}
\author[F. Alghamdi]{Fatimah Alghamdi}
\address{F. Alghamdi: Department of Mathematics, Faculty of Science, Jeddah University, 21589 Jeddah, Saudi Arabia}
\email{fmalghamdi@uj.edu.sa}
\author[R. Al-Ghefari]{Reem Al-Ghefari}
\address{R. Al-Ghefari: Department of Mathematics, Faculty of Science for Girls, King Abdulaziz University, 21589 Jeddah, Saudi Arabia}
\email{ralghefari@kau.edu.sa}

\begin{document}
\begin{abstract}
In this paper, we study warped products of contact skew-CR submanifolds, called contact skew CR-warped products. We establish a lower bound relationship between the squared norm of the second fundamental form and the warping function. The equality case of the inequality is investigated and some special cases of derived inequality are given. Furthermore, we provide non-trivial examples of such submanifolds.
\end{abstract}
\maketitle

\section{Introduction}
\label{intro}

The concept of skew CR-submanifolds of almost Hermitian manifolds was given by G. S. Ronsse \cite{Ronsse} to unify and generalize the concepts of holomorphic, totally real, CR, slant, semi-slant and pseudo-slant (hemi-slant in the sense of B. Sahin \cite{Sahin1}) submanifolds by exploiting the behavior of the bounded symmetric linear operator. Later, this idea is extended to the contact geometry by Tripathi in \cite{Tripathi} with the name of almost semi-invariant submanifolds as a generalized class of invariant, anti-invariant, slant, contact CR, bi-slant submanifolds of contact metric manifolds.

 On the other hand, the warped products of skew CR-submanifolds of Kaehler manifolds were studied by B. Sahin in \cite{Sahin2} as a generalization of CR-warped products introduced by B.-Y. Chen in his seminal work \cite{C3,C4,C6,C7} and of warped product hemi-slant submanifolds, studied by B. Sahin in \cite{Sahin1}. Later on, the contact version of skew CR-warped products of cosymplectic manifolds appeared in \cite{Haider}. Recently, we studied warped product skew CR-submanifolds of Kenmotsu manifolds in \cite{Monia}. For up-to-date survey on warped product manifolds and warped product submanifolds we refer to B.-Y. Chen's books \cite{book,book17} and his survey article \cite{C5}.

In this paper, we study the contact skew CR-warped product submanifolds by considering the base manifold is the Riemannian product of invariant and proper slant submanifolds of a Sasakian manifold and the fiber of warped product is an anti-invariant submanifold.

The paper is organized as follows: In Section 2, we give some basic formulas and definitions for almost contact metric manifolds and their submanifolds. In Section 3, we recall the definition of skew CR-submanifolds and provide two non-trivial examples. In this section, we also find some useful relations for contact skew CR-warped products those are essential to derive our main result. In Section 4, we derive a lower bound relation for the squared norm of the second fundamental form in terms of components of the gradient of warping function along both factors of a base manifold. The equality case is also considered. In Section 5, we give some special cases of our derived inequality. In Section 6, we give two non-trivial examples of skew CR-warped products in Euclidean spaces.

\section{Preliminaries}
A $(2m +1)$-dimensional differentiable manifold $\tilde M$ is called an {\textit{almost contact manifold}} if there is an almost contact structure $(\varphi, \xi, \eta)$ consisting of a $(1, 1)$ tensor field $\varphi$, a vector field $\xi$ and a $1$-form $\eta$ satisfying \cite{Blair}
\begin{align}
\label{struc}
\varphi^2=-I+\eta\otimes\xi,\,\,\eta(\xi)=1,\,\,\,\varphi\xi=0,\,\,\,\eta\circ\varphi=0,
\end{align}
where $I:T\tilde M\to T\tilde M$ is the identity mapping. From the definition it follows that the $(1, 1)$-tensor field $\varphi$ has constant rank $2m$ (cf. \cite{Blair}). An almost contact manifold $(\tilde M, \varphi, \eta, \xi)$ is said to be { \textit{normal}} when the tensor field $N_{\varphi}=[\varphi, \varphi]+2d\eta\otimes\xi$ vanishes identically, where $[\varphi, \varphi]$ is the Nijenhuis torsion of $\varphi$. It is known that any almost contact manifold $(\tilde M, \varphi, \eta , \xi)$ admits a Riemannian metric $\tilde g$ such that
\begin{align}
\label{metric}
\tilde g(\varphi X, \varphi Y)=\tilde g(X, Y)-\eta(X)\eta(Y)
\end{align}
for any $X, Y\in\Gamma(T\tilde M)$, where the $\Gamma(T\tilde M)$ is the Lie algebra of vector fields on $\tilde{M}$. This metric $\tilde g$ is called a { \textit{compatible metric}} and the manifold $\tilde M$ together with the structure $(\varphi, \xi, \eta, \tilde g)$ is called an { \textit{almost contact metric manifold}}. As an immediate consequence of (\ref{metric}), one has $\eta(X)=\tilde g(X,\xi),\,\,\eta(\xi)=1$ and $\tilde g(\varphi X, Y)=-\tilde g(X, \varphi Y)$. Hence the fundamental 2-form $\Phi$ of $\tilde{M}$ is defined $\Phi(X, Y)=\tilde g(X, \varphi Y)$ and the manifold is said to be { \textit{contact metric manifold}} if $\Phi=d\eta$. If $\xi$ is a Killing vector field with respect to $\tilde g$, the contact metric structure is called a $K-${ \textit{contact structure}}. A normal contact metric manifold is said to be a {\textit{Sasakian manifold}}. An almost contact metric manifold is Sasakian if and only if
\begin{align}
\label{struc1}
(\tilde\nabla_{X}\varphi)Y=\tilde g(X, Y)\xi-\eta(Y)X
\end{align}
for all $X, Y\in\Gamma(T\tilde M)$, where $\tilde\nabla$ is the Levi-Civita connection of $\tilde g$. From the formula (\ref{struc1}), it follows that $\tilde\nabla_{X}\xi=-\varphi X$. A Sasakian manifold is always a $K-$contact manifold and the converse is true in the dimension three.

Let M be a submanifold of a Riemannian manifold $\tilde M$ equipped with a Riemannian metric $\tilde g$. We use the same symbol $g$ for both the metrics $\tilde g$ of  $\tilde M$ and the induced metric $g$ on the submanifold $M$. Let $\Gamma(TM)$ the Lie algebra of vector fields on $M$ and ${\Gamma(T^{\perp}M)}$, the set of all vector fields normal to $M$. If we denote by $\nabla$, the Levi-Civita connection of $M$, then the Gauss and Weingarten formulas are respectively given by
\begin{align}
\label{gauss}
\tilde\nabla_XY=\nabla_XY+\sigma(X, Y),
\end{align}
\begin{align}
\label{Wn}
\tilde\nabla_XN=-A_NX+\nabla^{\perp}_XN,
\end{align}
for any vector field $X,~Y\in\Gamma(TM)$ and $N\in\Gamma(T^\perp M)$, where $\nabla^{\perp}$  is the normal connection in the normal bundle, $\sigma$ is the second fundamental form and $A_N$ is the shape operator (corresponding to the normal vector field $N$) for the immersion of $M$ into $\tilde M$. They are related by $g(\sigma(X,Y), N)=g(A_NX, Y)$. 

A submanifold $M$ is said to be totally geodesic if $\sigma=0$ and totally umbilical if $\sigma(X, Y)=g(X, Y)H,\,\,\forall\,X,\,Y\in\Gamma(TM)$, where $H=\frac{1}{n}{\sum}_{i=1}^{n}\sigma(e_i, e_i)$ is the mean curvature vector of $M$. For any $x\in M$ and $\{e_1,\cdots,e_n,\cdots,e_{2m+1}\}$ is an orthonormal frame of $T_x\tilde M$ such that $e_1,\cdots,e_n$ are tangent to $M$ at $x$. Then, we set
\begin{align}
\label{second1}
\sigma_{ij}^r=g(\sigma(e_i, e_j), e_r),~~i,j\in\{1,\cdots,n\},~~r\in\{n+1,\cdots,2m+1\},
\end{align}
\begin{align}
\label{second2}
\|\sigma\|^2=\sum_{i,j=1}^{n}g(\sigma(e_i, e_j), \sigma(e_i, e_j)).
\end{align}
According to the behaviour of the tangent bundle of a submanifold under the action of the almost contact structure tensor $\varphi$ of the ambient manifold, there are two well-known classes of submanifolds, namely, $\varphi$-invariant submanifolds and $\varphi$-anti-invariant submanifolds. In the first case the tangent space of the submanifold remains invariant under the action of the almost contact structure tensor $\varphi$ whereas in the second case it is mapped into the normal space.

Later, A. Bejancu \cite{Bejancu} generalized the concept of invariant and anti-invariant submanifolds in to a semi-invariant submanifold (also known as contact CR-submanifold \cite{Hasegawa}, \cite{Yano}). A submanifold $M$ tangent to the structure vector field $\xi$ of an almost contact metric manifold $\tilde M$ is called a {\textit{contact CR-submanifold}} if there exists a pair of orthogonal distributions ${\mathfrak{D}}:x\to {\mathfrak {D}}_x$ and ${\mathfrak{D}}^\perp:x\to {\mathfrak{D}}^\perp_x$,~$\forall~x\in M$ such that $TM=\mathfrak D\oplus\mathfrak D^\perp\oplus\langle\xi\rangle$, where $\langle\xi\rangle$ is the 1-dimensional distribution spanned by the structure vector field $\xi$ with $\mathfrak D$ is invariant, i.e., $\varphi\mathfrak D=\mathfrak D$ and $\mathfrak D^\perp$ is anti-invariant, i.e., $\varphi{\mathfrak D^\perp}\subseteq T^\perp M$. Obviously, invariant and anti-invariant submanifolds are contact CR-submanifolds with $\mathfrak D^\perp = \{0\}$ and $\mathfrak D = \{0\}$, respectively.

Slant submanifolds in complex geometry were defined and studied by B.-Y. Chen \cite{C1,C2}. In \cite{Lotta}, A. Lotta introduced the contact version of slant submanifolds. Let $M$ be a submanifold of an almost contact metric manifold $\tilde M$. Let $\mathfrak D$ be a differentiable distribution on $M$. For any non-zero vector $X\in\mathfrak D_x$, the angle $\theta_{\mathfrak D}(X)$ between $\varphi X$ and $\mathfrak D_x$ is a slant angle of $X$ with respect to the distribution $\mathfrak D$. If the slant angle $\theta_{\mathfrak D}(X)$ is constant, i.e., it is independent of the choice $x\in M$ and $X\in\mathfrak D_x$, then $\mathfrak D$ is called a $\theta$-slant distribution and $\theta_{\mathfrak D}(X)=\theta_{\mathfrak D}$ is called the slant angle of the distribution $\mathfrak D$. A submanifold $M$ tangent to $\xi$ is said to be { \textit{slant}} if for any $x\in M$ and any $X\in T_xM$, linearly independent to $\xi$, the angle
between $\varphi X$ and $T_xM$ is a constant $\theta\in[0, \pi/2]$, called the { \textit{slant angle}} of $M$ in $\tilde M$.
Invariant and anti-invariant submanifolds are $\theta$-slant submanifolds with slant angle $\theta=0$ and $\theta=\pi/2$, respectively. A slant submanifold which is neither invariant nor anti-invariant is called { \textit{proper slant}}. For more details, we refer to \cite{C2,Cab2}.

For any vector field $X\in\Gamma(TM)$, we have
\begin{align}
\label{phi}
\varphi X=TX+FX,
\end{align}
where $TX$ and $FX$ are the tangential and normal components of $\varphi X$, respectively. For a slant submanifold of almost contact metric manifolds we have the following useful result.
\begin{theorem}\label{slt1} \cite{Cab2} Let $M$ be a submanifold of an almost contact metric manifold $\tilde M$, such that $\xi\in\Gamma(TM)$. Then $M$ is slant if and only if there exists a constant $\lambda\in[0,1]$ such that
\begin{align}
\label{slant}
T^2=\lambda(-I+\eta\otimes\xi).
\end{align}
Furthermore, if $\theta$ is slant angle, then $\lambda=\cos^2\theta$.
\end{theorem}

Following relations are straightforward consequence of (\ref{slant})
\begin{align}
\label{slant1}
g(TX,TY)=\cos^2\theta[g(X,Y)-\eta(X)\eta(Y)]
\end{align}
\begin{align}
\label{slant2}
g(FX, FY) =\sin^2\theta[g(X, Y)-\eta(X)\eta(Y)]
\end{align}
for any $X,Y\in\Gamma (TM).$

Beside these classes of submanifolds of almost contact metric manifolds there are some other submanifolds. J.L. Caberizo et al. defined and studied semi-slant submanifolds of Sasakian manifolds in \cite{Cab1}. A submanifold $M$ of an almost contact metric manifold $\tilde M$ is said to be a { \textit{semi-slant submanifold}} if there exists a pair of orthogonal distributions ${\mathfrak{D}}$ and ${\mathfrak{D}}^\theta$ on $M$ such that ${\mathfrak{D}}$ is $\varphi$-invariant and ${\mathfrak{D}}^\theta$ is proper slant with slant angle $\theta$ with $TM={\mathfrak{D}}\oplus{\mathfrak{D}}^\theta\oplus\langle\xi\rangle$.

Pseudo-slant submanifolds were defined by Carriazo in \cite{Carriazo} under the name of { \textit{anti-slant submanifolds}} as a particular class of bi-slant submanifolds. Later, he called these classes of submanifolds as { \textit{pseudo-slant submanifolds}}.  A submanifold $M$ of an almost contact metric manifold $\tilde M$ is said to be a { \textit{pseudo-slant submanifold}} if there exists a pair of orthogonal distributions ${\mathfrak{D}}^\perp$ and ${\mathfrak{D}}^\theta$ on $M$ such that $TM={\mathfrak{D}}^\perp\oplus{\mathfrak{D}}^\theta\oplus\langle\xi\rangle$ with ${\mathfrak{D}}^\perp$ is anti-invariant, that is,  $\varphi({\mathfrak{D}}^\perp)\subset T^\perp M$ and ${\mathfrak{D}}^\theta$ is a proper slant distribution with angle $\theta.$

\section{Contact skew CR-warped product submanifolds}
Skew CR-submanifolds introduced by Ronsse \cite{Ronsse} for almost Hermitian manifolds. Later, for contact metric manifolds, Tripathi \cite{Tripathi} studied contact skew CR-submanifolds under the name almost semi-invariant submanifolds by exploiting the behaviour of a natural bounded symmetric linear operator $T^2=Q$ on the submanifold. From \eqref{metric} and \eqref{phi}, it is easy to see that $g(TX, Y)=-g(X, TY)$, for any $X, Y\in\Gamma(TM)$, which implies that $g(QX, Y)=g(X, QY)$, i.e., $Q$ is a symmetric operator, therefore its eigenvalues are real and diagonalizable. Moreover, its eigenvalues are bounded by $-1$ and $0$.

Since $\xi\in\Gamma(TM)$, then we have $TM=\langle\xi\rangle\oplus\langle\xi\rangle^\perp$ where $\langle\xi\rangle$ is the distribution spanned by $\xi$ and $\langle\xi\rangle^\perp$ is the orthogonal complementary distribution of $\langle\xi\rangle$ in $M$. For any $x\in M$, we may write
\begin{equation*}
 {\mathfrak{D}}^{\lambda}_x=ker\left(Q+\lambda^2(x)I\right)_x,
 \end{equation*}
where I is the identity transformation and $\lambda(x)\in [0,1]$ such that $-\lambda^2(x)$ is an eigenvalue of $Q(x)$. We note that ${\mathfrak{D}}^1_x=ker F$ and ${\mathfrak{D}}^0_x=ker T$. ${\mathfrak{D}}^1_x$ is the maximal $\varphi$-invariant subspace of $T_xM$ and ${\mathfrak{D}}^0_x$ is the maximal $\varphi$-anti-invariant subspace of $T_xM$. From now on, we denote the distributions ${\mathfrak{D}}^1$ and ${\mathfrak{D}}^0$ by ${\mathfrak{D}}$ and ${\mathfrak{D}}^{\perp}$, respectively. Since $Q_x$ is symmetric and diagonalizable, for some integer $k$ if $-\lambda^2_1(x),\cdots,-\lambda^2_k(x)$ are the eigenvalues of $Q$ at $x\in M$, then $\langle\xi\rangle^\perp_x$ can be decomposed as direct sum of mutually orthogonal eigenspaces, i.e.
\begin{equation*}
\langle\xi\rangle^\perp_x={\mathfrak{D}}^{\lambda_1}_x\oplus{\mathfrak{D}}^{\lambda_2}_x\cdots\oplus{\mathfrak{D}}^{\lambda_k}_x.
\end{equation*}
Each ${\mathfrak{D}}^{\lambda_i}_x$, $1\leq i\leq k$, is a $T$-invariant subspace of $T_xM$. Moreover if $\lambda_i\neq 0$, then ${\mathfrak{D}}^{\lambda_i}_x$
is even dimensional. We say that a submanifold $M$ of an almost contact metric manifold $\tilde M$ is a generic submanifold if there exists an integer $k$ and functions $\lambda_i$, $1\leq i\leq k$ defined on $M$ with values in $(0,1)$ such that 
\begin{enumerate}
 \item [(1)]
Each $-\lambda^2_i(x),~~1\leq i\leq k$ is a distinct eigenvalue of $Q$ with 
\begin{equation*}
T_xM={\mathfrak{D}}_x\oplus{\mathfrak{D}}^{\perp}_x\oplus{\mathfrak{D}}^{\lambda_1}_x\oplus\cdots\oplus{\mathfrak{D}}^{\lambda_k}_x\oplus\langle\xi\rangle_x
\end{equation*}
for any $x\in M$.
\item [(2)]
The dimensions of ${\mathfrak{D}}_x,~{\mathfrak{D}}^{\perp}_x$ and ${\mathfrak{D}}^{\lambda_i}$,$1\leq i\leq k$ are independent on $x\in M$.
\end{enumerate}

Moreover, if each $\lambda_i$ is constant on $M$, then $M$ is called a skew CR-submanifold. Thus, we observe that CR-submanifolds are a particular class of skew CR-submanifolds with $k=0$, ${\mathfrak{D}}\neq \{0\}$ and ${\mathfrak{D}}^{\perp}\neq \{0\}$. And slant submanifolds are also a particular class of skew CR-submanifolds with $k=1$, ${\mathfrak{D}}=\{0\}$, ${\mathfrak{D}}^{\perp}=\{0\}$ and $\lambda_1$ is constant. Moreover, if ${\mathfrak{D}}^{\perp}=\{0\}$, ${\mathfrak{D}}\neq {0}$ and $k=1$, then $M$ is a semi-slant submanifold. Furthermore, if ${\mathfrak{D}}=\{0\}$, ${\mathfrak{D}}^{\perp}\neq \{0\}$ and $k=1$, then $M$ is a pseudo-slant (or hemi-slant) submanifold.

A submanifold $M$ of an almost contact metric manifold $\tilde M$ is said to be a contact skew CR-submanifold of order $1$ if $M$ is a skew CR-submanifold such that $k=1$ and $\lambda_1$ is constant. In this case, the tangent bundle of $M$ is decomposed as\\
\begin{equation*}
TM={\mathfrak{D}}\oplus {\mathfrak{D}}^{\perp}\oplus{\mathfrak{D}}^{\theta}\oplus\langle\xi\rangle
\end{equation*}

The normal bundle $T^\perp M$ of a contact skew CR-submanifold $M$ is decomposed as
\begin{align*}T^\perp M=\varphi{\mathfrak D^{\perp}}\oplus F{\mathfrak D}^{\theta}\oplus{\mathfrak\mu},
\end{align*}
where ${\mathfrak{\mu}}$ is a $\varphi$-invariant normal subbundle of $T^\perp M$.

We provide the following examples of contact skew CR-submanifolds of order $1$ in the Euclidean spaces.

\begin{example}\label{ex1}\rm{Consider the Euclidean $11$-space ${\mathbb{R}}^{11}$ with cartesian coordinates $(x_1,\cdots,x_5,\,y_1,\cdots,y_5,\,t)$ and the standard Euclidean metric $<\,,\,>$. Define the almost contact structure on ${\mathbb{R}}^{11}$ as follows:
\begin{align*}
\varphi\left(\frac{\partial}{\partial x_i}\right)=-\left(\frac{\partial}{\partial y_i}\right),\,\,\varphi\left(\frac{\partial}{\partial y_j}\right)=\left(\frac{\partial}{\partial x_j}\right),\,\,\varphi\left(\frac{\partial}{\partial t}\right)=0;\,\,1\leq i, j\leq5.
\end{align*}
Then, it is easy to see that $({\mathbb{R}}^{11}, \varphi, \xi, \eta, <\,,\,>)$ is an almost contact metric manifold with $\xi=\frac{\partial}{\partial t}$ and $\eta=dt$. Let $M$ be a submanifold of ${\mathbb{R}}^{11}$ defined by the immersion $\psi:M\to{\mathbb{R}}^{11}$ as follows:
\begin{align*}
\psi(u, v, w, r, s, t)=(u+v,\,\cosh w,\,kr,\,\cos r,\,\cos s,\,u-v,\,\sinh w,\,s,\,\sin r,\,\sin s,\,t)
\end{align*}
for any non-zero constant $k$. Then the tangent space of $M$ is spanned by the following vectors:
\begin{align*}
&X_1=\frac{\partial}{\partial x_1}+\frac{\partial}{\partial y_1},\,\,X_2=\frac{\partial}{\partial x_1}-\frac{\partial}{\partial y_1},\,\,X_3=\sinh w\frac{\partial}{\partial x_2}+\cosh w\frac{\partial}{\partial y_2},\\
&X_4=k\frac{\partial}{\partial x_3}-\sin r\frac{\partial}{\partial x_4}+\cos r\frac{\partial}{\partial y_4},\,\,X_5=-\sin s\frac{\partial}{\partial x_5}+\frac{\partial}{\partial y_3}+\cos s\frac{\partial}{\partial y_5},\,\,X_6=\frac{\partial}{\partial t}.
\end{align*}
Hence, we find that $\varphi X_3$ is orthogonal to $TM$, thus ${\mathfrak{D}}^\perp=\rm{Span}\{X_3\}$ is an anti-invariant distribution and ${\mathfrak{D}}=\rm{Span}\{X_1, X_2\}$ is an invariant distribution; while ${\mathfrak{D}}^\theta=\rm{Span}\{X_4, X_5\}$ is a slant distribution with slant angle $\theta=\cos^{-1}\left(\frac{k}{\sqrt{2(1+k^2)}}\right)$. Hence, the tangent space is decomposed $TM={\mathfrak{D}}\oplus{\mathfrak{D}}^\perp\oplus{\mathfrak{D}}^\theta\oplus<\xi>$, i.e., $M$ is a contact skew CR-submanifold of order $1$.}
\end{example}
\begin{example}\label{ex2}\rm{Let $M$ be a submanifold ${\mathbb{R}}^{9}$ given by
\begin{align*}
&x_1=u,,\,y_1=-v,\,x_2=r,\,y_2=s,\;x_3=s\cos\theta,\,y_3=s\sin\theta,\\
&x_4=\cos w,\,y_4=-\sin w, t=t.
\end{align*}
It is easy to find that the local frame of $TM$ is spanned by
\begin{align*}
&X_1=\frac{\partial}{\partial x_1},\,\,X_2=-\frac{\partial}{\partial y_1},\,\,X_3=\frac{\partial}{\partial x_2},\\
&X_4=\cos\theta\frac{\partial}{\partial x_3}+\frac{\partial}{\partial y_2}+\sin\theta\frac{\partial}{\partial y_3},\,\,X_5=-\sin w\frac{\partial}{\partial x_4}-\cos w\frac{\partial}{\partial y_4},\,\,X_6=\frac{\partial}{\partial t}.
\end{align*}
Then, using the almost contact structure of ${\mathbb{R}}^{9}$ defined in Example \ref{ex1}, we find that $\varphi X_5$ is orthogonal to $TM$, thus ${\mathfrak{D}}^\perp=\rm{Span}\{X_5\}$ is an anti-invariant distribution and ${\mathfrak{D}}=\rm{Span}\{X_1, X_2\}$ is an invariant distribution; while ${\mathfrak{D}}^\theta=\rm{Span}\{X_3, X_4\}$ is a slant distribution with slant angle $\theta=45^{\circ}$. Hence, $M$ is a contact skew CR-submanifold of order $1$.}
\end{example}

Let $\left(B,g_{B} \right)$ and $\left(F ,g_{F} \right)$ be two Riemannian manifolds and ${f}$ be a positive smooth function on $B$. Consider the product manifold $B\times F$ with canonical projections $\pi_1:B \times F\to B\quad{\rm and}\quad \pi_2:B \times F\to F$. Then the manifold $M=B \times_{f} F $ is said to be \textit{warped product} if it is equipped with the following warped metric
\begin{align}\label{wp1}
g(X,Y)=g_{B}\left(\pi_1{\ast}(X),\pi_1{\ast}(Y)\right) +(f\circ\pi_1)^{2}g_{F}\left(\pi_2{\ast}(X),\pi_2{\ast}(Y)\right)
\end{align}
for all $X,Y\in \Gamma(TM)$ and `$\ast$' stands for derivation maps.
The function $f$ is called { \textit the warping function} and a warped product manifold $M$ is said to be { \textit trivial} or simply a Riemannian product of $B$ and $F$ if $f$ is constant.

\begin{proposition}\label{wpp1}\cite{Bishop}
For $X, Y \in \Gamma(TB)$ and $Z, W \in \Gamma(TF)$, we obtain for the warped product manifold $M=B\times_{f} F$ that
\begin{itemize}
\item[(i)]   $\nabla _{X}Y \in \Gamma(TB),$
\item[(ii)]	$\nabla _{X}Z =\nabla _{Z}X=X(\ln f)Z,$
\item[(iii)]	$\nabla _{Z}W =\nabla^\prime _{Z}W -\frac{g(Z, W)}{f} \vec\nabla f,$
\end{itemize}
where $\nabla$ and $\nabla^\prime$ denote the Levi-Civita connections on $M$ and $F$, respectively and $\vec\nabla f$ is the gradient of $f$ defined by $g(\vec\nabla f, X)=X(f)$.\end{proposition}

\begin{remark}\label{wpr1}
It is also important to note that for a warped product $M=B \times_{f}F$; $B$ is totally geodesic and $F$ is totally umbilical in $M$ \cite{Bishop,C3}. 
\end{remark}

In this section, we study warped products of contact skew CR-submanifolds of order $1$ of a Sasakian manifold $\tilde M$ which we define as: A warped product submanifolds of the form $M=B\times_fM_\perp$ is called a {\textit{contact skew CR-warped product submanifold}} if $B=M_T\times M_\theta$ is the product of $M_T$ and $M_\theta$, called semi-slant product, where $M_T,\, M_\perp$ and $M_\theta$ are invariant, anti-invariant and proper slant submanifolds of $\tilde M$, respectively. Throughout this paper, we assume the structure vector field $\xi$ tangent to the submanifold. For this reason, on a contact skew CR-warped product $M=B\times_fM_\perp$, two case arise either $\xi$ is tangent to $M_\perp$ or $\xi$ is tangent to $B$. When, $\xi\in\Gamma(TM_\perp)$, then we have the following non-existence result.

\begin{theorem}\label{wpt1}
 Let $M=B\times_fM_\perp$ be a contact skew CR-warped product submanifold with $B=M_T\times M_\theta$ of a Sasakian manifold $\tilde M$ such that $\xi$ is tangent to $M_\perp$. Then $M$ is simply a Riemannian product submanifold of $\tilde M$.
\end{theorem}
\begin{proof} For any $U_1+U_2=U\in\Gamma(TB)$,  where $U_1\in\Gamma(TM_T)$ and $U_2\in\Gamma(TM_\theta)$, we have
\begin{align*}
\tilde\nabla_U\xi=-\phi U=-\phi U_1-TU_2-FU_2.
\end{align*}
Using \eqref{gauss} and equating the tangential components, we derive
\begin{align*}
\nabla_U\xi=-\phi U_1-TU_2.
\end{align*}
Then using Proposition \ref{wpp1} (ii), we get
\begin{align*}
U(\ln f)\xi=-\phi U_1-TU_2.
\end{align*}
Taking the inner product with $\xi$ in the above relation, we find that $U(\ln f)=0$, i.e., $f$ is constant, which proves the theorem completely.
\end{proof}

From now, for the simplicity we denote the tangent spaces of $M_T,\, M_\perp$ and $M_\theta$ by the same symbols $\mathfrak{D},\,\,\mathfrak{D}^\perp$ and $\mathfrak{D}^\theta$, respectively. 

Now, if we consider $\xi\in\Gamma(TB)$, then there are two possibilities that either $\xi$ is tangent to $M_T$ or tangent to $M_\theta$. For this, we have the following useful results.

\begin{lemma}\label{wpl1}
Let $M=B\times{_f}M_\perp$ be a contact skew CR- warped product submanifold of order $1$ of a Sasakian manifold $\tilde M$ such that $\xi$ is tangent to $B$ and $B=M_T\times M_{\theta}$, where $M_T$ and $M_\theta$ are invariant and proper slant submanifolds of $\tilde M$, respectively. Then, we have
\begin{enumerate}
\item [(i)]
$\xi(\ln f)=0,$
\item [(ii)]
$g(\sigma(X, Y),\varphi Z)=0,$
\item [(iii)]
$g(\sigma(X, V), \varphi Z)=-g(\sigma(X, Z), FV)=0,$
\end{enumerate}
for any $X,Y\in \Gamma(\mathfrak{D})$, $V\in \Gamma(\mathfrak{D}^\theta)$ and $Z\in \Gamma(\mathfrak{D}^\perp)$.
\end{lemma}
\begin{proof} For any $Z\in \Gamma(\mathfrak{D}^\perp)$, we have $\tilde\nabla_Z\xi=-\varphi Z,$ by using \eqref{gauss}, we find that $\nabla_Z\xi=0,\,\,\sigma(Z, \xi)=-\varphi Z.$ Using Proposition \ref{wpp1}, we get the first part of the lemma. For the second part, we have
\begin{align*}
g(\sigma(X, Y), \varphi Z)=g(\tilde\nabla_XY, \varphi Z)=-g(\tilde\nabla_X\varphi Y, Z)+g((\tilde\nabla_X\varphi)Y, Z).
\end{align*}
for any $X,Y\in \Gamma(\mathfrak{D})$ and $Z\in \Gamma(\mathfrak{D}^\perp)$. Using \eqref{struc1} and the orthogonality of vector fields, we derive
\begin{align*}
g(\sigma(X, Y), \varphi Z)=g(\tilde\nabla_XZ, \varphi Y)=g(\nabla_XZ, \varphi Y).
\end{align*}
Again, using Proposition \ref{wpp1}, we find that $g(\sigma(X, Y), \varphi Z)=X(\ln f)g(Z, \varphi Y)=0$, which is (ii). Similarly, for any $X\in \Gamma(\mathfrak{D})$, $V\in \Gamma(\mathfrak{D}^\theta)$ and $Z\in \Gamma(\mathfrak{D}^\perp)$, we have
\begin{align*}
g(\sigma(X, V), \varphi Z)=g(\tilde\nabla_XV, \varphi Z)=-g(\tilde\nabla_X\varphi V, Z)+g((\tilde\nabla_X\varphi)V, Z).
\end{align*}
Again, from \eqref{struc1}, \eqref{phi} and the orthogonality of vector fields, we obtain
\begin{align*}
g(\sigma(X, V), \varphi Z)=-g(\tilde\nabla_XTV, Z)+g(\tilde\nabla_XFV, Z)=g(\nabla_XZ, TV)-g(A_{FV}X, Z).
\end{align*}
Then from Proposition \ref{wpp1}, we get $g(\sigma(X, V), \varphi Z)=X(\ln f)g (Z, TV)-g(\sigma(X, Z), FV)$. Hence, by the orthogonality of vector fields, the second term vanishes identically which gives the first equality of (iii). On the other hand, for any $X\in \Gamma(\mathfrak{D})$, $V\in \Gamma(\mathfrak{D}^\theta)$ and $Z\in \Gamma(\mathfrak{D}^\perp)$, we have
\begin{align*}
g(\sigma(X, V), \varphi Z)=g(\tilde\nabla_VX, \varphi Z)=-g(\tilde\nabla_V\varphi X, Z)+g((\tilde\nabla_V\varphi)X, Z).
\end{align*}
Again, using the structure equation of Sasakian manifold, the orthogonality of vector fields and Proposition \ref{wpp1}, we get $g(\sigma(X, V), \varphi Z)=0,$ which is the second equality. Hence, the proof is complete.
\end{proof}

\begin{lemma}\label{wpl2}
Let $M=B\times{_f}M_\perp$ be a contact skew CR-warped product submanifold of order $1$ of a Sasakian manifold $\tilde M$ such that $\xi$ is tangent to $B$. Then
\begin{align}\label{wp3}
g(\sigma(U, V), \varphi Z)=g(\sigma(U, Z), FV)
\end{align}
for any $U, V\in \Gamma(\mathfrak{D}^\theta)$ and $Z\in \Gamma(\mathfrak{D}^\perp)$.
\end{lemma}
\begin{proof} For any $U,V\in \Gamma(\mathfrak{D}^\theta)$ and $Z\in \Gamma(\mathfrak{D}^\perp)$, we have
\begin{align*}
g(\sigma(U, V), \varphi Z)=g(\tilde\nabla_UV, \varphi Z)=-g(\tilde\nabla_U\varphi V, Z)+g((\tilde\nabla_U\varphi)V, Z).
\end{align*}
Using \eqref{struc1}, \eqref{phi} and the orthogonality of vector fields, we find
\begin{align*}
g(\sigma(U, V), \varphi Z)=-g(\tilde\nabla_UTV, Z)-g(\tilde\nabla_UFV, Z)=g(\nabla_UZ, TV)+g(A_{FV}U, Z).
\end{align*}
By Proposition \ref{wpp1} and the orthogonality of vector field, we obtain $g(\sigma(U, V), \varphi Z)=g(\sigma(U, Z), FV)$, which proves the lemma completely.
\end{proof}
\begin{lemma}\label{wpl3}
Let $M=B\times{_f}M_\perp$ be a contact skew CR-warped product submanifold of order $1$ of a Sasakian manifold $\tilde M$ such that $\xi$ is tangent to $B$. Then, we have
\begin{align}\label{wp4}
g(\sigma(\varphi X, Z), \varphi W)=X(\ln f)g(Z, W)
\end{align}
for any $X\in \Gamma(\mathfrak{D})$ and $Z, W\in \Gamma(\mathfrak{D}^\perp)$.
\end{lemma}
\begin{proof} For any $X\in \Gamma(\mathfrak{D})$ and $Z, W\in \Gamma(\mathfrak{D}^\perp)$, we have
\begin{align*}
g(\sigma(X, Z), \varphi W)=g(\tilde\nabla_ZX, \varphi W)=-g(\tilde\nabla_Z\varphi X, W)+g((\tilde\nabla_Z\varphi)X, W).
\end{align*}
Using Proposition \ref{wpp1}, structure equation \eqref{struc1} and the orthogonality of vector fields, we find
\begin{align}\label{wp5}
g(\sigma(X, Z), \varphi W)=-\varphi X(\ln f)g(Z, W)-\eta(X)g(Z, W).
\end{align}
Interchanging $X$ by $\varphi X$ and using \eqref{struc}, we find \eqref{wp4}, which completes the proof.
\end{proof}

A warped product $M=B\times_fF$ is said to be {\textit{mixed totally geodesic}} if $\sigma(X, Z)=0$, for any $X\in\Gamma(TB)$ and $Z\in\Gamma(TF)$. From Lemma \ref{wpl3}, we have the following consequence for a mixed totally geodesic warped product.

\begin{theorem}\label{wpt2}
Let $M=B\times{_f}M_\perp$ be a $\mathfrak{D}-\mathfrak{D}^\perp$ mixed totally geodesic contact skew CR-warped product submanifold of order $1$ of a Sasakian manifold $\tilde M$ such that $\xi$ is tangent to $B$. Then $M$ is simply a Riemannian product manifold.
\end{theorem}
\begin{proof} The proof of this theorem follows from \eqref{wp4} and the mixed totally geodesic condition.\end{proof}

\begin{lemma}\label{wpl4}
Let $M=B\times{_f}M_\perp$ be a contact skew CR- warped product submanifold of order $1$ of a Sasakian manifold $\tilde M$ such that $\xi$ is tangent to $B$. Then
\begin{enumerate}
\item [(i)]
$g(\sigma(Z, W), FV)-g(\sigma(Z, V), \varphi W)=\left(TV(ln f)+\eta(V)\right)g(Z, W),$
\item [(ii)]
$g(\sigma(Z, W), FTV)-g(\sigma(Z, TV), \varphi W)=-\cos^2\theta\,V(ln f)\,g(Z, W)$
\end{enumerate}
for any $Z, W\in \Gamma(\mathfrak{D}^\perp)$ and $V\in \Gamma(\mathfrak{D}^\theta)$.
\end{lemma}
\begin{proof} For any $V\in \Gamma(\mathfrak{D}^\theta)$ and $Z, W\in \Gamma(\mathfrak{D}^\perp)$, we have
\begin{align*}
g(\sigma(Z, V), \varphi W)=g(\tilde\nabla_ZV, \varphi W)=-g(\tilde\nabla_Z\varphi V, W)+g((\tilde\nabla_Z\varphi)V, W).
\end{align*}
Using \eqref{struc1} and \eqref{phi}, we derive
\begin{align*}
g(\sigma(Z, V), \varphi W)=-g(\tilde\nabla_ZTV, W)-g(\tilde\nabla_ZFV, W)-\eta(V)g(Z, W),
\end{align*}
which on using Proposition \ref{wpp1} (ii) implies that
\begin{align*}
g(\sigma(Z, W), FV)-g(\sigma(Z, V), \varphi W)=\left(TV(\ln f)+\eta(V)\right)g(Z, W),
\end{align*}
which is (i). Interchanging $V$ by $TV$ in (i) and using Theorem \ref{slt1}, we find (ii), which ends the proof.
\end{proof}

\section{Inequality for $\|\sigma\|^2$}
Let $M=B\times_fM_\perp$ be a $n-$dimensional contact skew CR-warped product submanifold of a $(2m+1)$-dimensional Sasakian manifold $\tilde M$ with $B=M_T\times M_\theta$ and $\xi$ is tangent to $B$. If $\dim M_T=m_1,\,\,\dim M_\perp=m_2$ and $\dim M_\theta=m_3$, then, clearly we have $n=m_1+m_2+m_3$. We denote the tangent bundle of $M_T, \,\,M_\perp$ and $M_\theta$ by $\mathfrak{D},\,\,\mathfrak{D}^\perp$ and $\mathfrak{D}^\theta$, respectively. Since, $\xi\in\Gamma(TB)$, then we have two cases: either $\xi\in\Gamma(\mathfrak{D})$ or $\xi\in\Gamma(\mathfrak{D}^\theta)$. If we consider $\xi\in\Gamma(\mathfrak{D})$ then we set the orthonormal frame fields of $M$ as follows: ${\mathfrak{D}}=\rm{Span}\{e_1,\cdots,e_p,\,e_{p+1}=\varphi e_1,\cdots,e_{2p}=\varphi e_p, e_{m_1}=e_{2p+1}=\xi\},\,\,{\mathfrak{D}}^\perp=\rm{Span}\{e_{m_1+1}=\bar e_1,\cdots,e_{m_1+m_2}=\bar e_{m_2}\}$ and ${\mathfrak{D}}^\theta=\rm{Span}\{e_{m_1+m_2+1}=e^*_1,\cdots,e_{m_1+m_2+s}=e^*_s,\,e_{m_1+m_2+s+1}=e^*_{s+1}=\sec\theta\,Te^*_1,\cdots,e_{n}=e^*_{m_3}=\sec\theta\,Te^*_s\}$. Then, the normal subbundles of $T^\perp M$ are spanned by  $\varphi{\mathfrak{D}}^\perp=\rm{Span}\{e_{n+1}=\tilde e_1=\varphi\bar e_1,\cdots,e_{n+m_2}=\tilde e_{m_2}=\varphi\bar e_{m_2}\},\,\,F{\mathfrak{D}}^\theta=\rm{Span}\{e_{n+m_2+1}=\tilde e_{m_2+1}=\csc\theta\,Fe^*_1,\cdots,e_{n+m_2+s}=\tilde e_{m_2+s}=\csc\theta\,Fe^*_s,\,e_{n+m_2+s+1}=\tilde e_{m_2+s+1}=\csc\theta\sec\theta\,FTe^*_{1},\cdots,e_{n+m_2+m_3}=\tilde e_{m_2+m_3}=\csc\theta\sec\theta\,FTe^*_s\}$ and $\mu=\rm{Span}\{e_{n+m_2+m_3+1}=\tilde e_{m_2+m_3+1},\cdots,e_{2m+1}=\tilde e_{2(m-m_2-m_3)-m_1+1}\}$.

Now, using the above orthonormal frame fields and some results of previous sections, we derive the following main result of this paper.

\begin{theorem}\label{wpt3}
Let $M=B\times_fM_\perp$ be a $\mathfrak{D}^\perp-\mathfrak{D}^\theta$ mixed totally geodesic contact skew CR-warped product submanifold of order 1 of a Sasakian manifold $\tilde M$. Then we have:
\begin{enumerate}
\item [(i)] If $\xi$ is tangent to $M_T$, then 
\begin{align*}
\|\sigma\|^2\geq 2m_2\left(\|\nabla^T(\ln f)\|^2+1\right)+m_2\cot^2\theta\,\|\nabla^\theta(\ln f)\|^2.
\end{align*}
\item [(ii)] If $\xi$ is tangent to $M_\theta$, then 
\begin{align*}
\|\sigma\|^2\geq2m_2\|\nabla^T(\ln f)\|^2+ m_2\cot^2\theta\,\|\nabla^\theta(\ln f)\|^2,
\end{align*}
where $m_2=\dim M_\perp$ and $\nabla^T(\ln f)$ and $\nabla^\theta(\ln f)$ are the gradient components along $M_T$ and $M_\theta$, respectively.
\item [(iii)] If the equality sign holds in above inequalities, then $B$ is totally geodesic and $M_\perp$ is a totally umbilical in $\tilde M$.
\end{enumerate}
\end{theorem}
\begin{proof} From the definition of the second fundamental from $\sigma$, we have
\begin{align*}
\|\sigma\|^2=\sum_{i, j=1}^ng(\sigma(e_i, e_j), \sigma(e_i, e_j))=\sum_{r=n+1}^{2m+1}\sum_{i, j=1}^ng(\sigma(e_i, e_j), e_r).
\end{align*}
According to the constructed frame filed, the above relation takes the from
\begin{align}\label{wp4.1}
\|\sigma\|^2&=\sum_{r=n+1}^{n+m_2} \sum_{i, j=1}^{n} g(\sigma(e_i, e_j), e_r)^2+\sum_{r=n+m_2+1}^{n+m_2+m_3} \sum_{i, j=1}^{n} g(\sigma(e_i, e_j), e_r)^2\notag\\
&+\sum_{r=n+m_2+m_3+1}^{2m+1} \sum_{i, j=1}^{n} g(\sigma(e_i, e_j), e_r)^2.
\end{align}
Leaving the last $\mu-$ components in \eqref{wp4.1}. Then, we can spilt the above relation for the orthogonal spaces as follows
\begin{align}\label{wp4.2}
\|\sigma\|^2&\geq\sum_{r=1}^{m_2} \sum_{i, j=1}^{m_1} g(\sigma(e_i, e_j), \tilde e_r)^2+2\sum_{r=1}^{m_2} \sum_{i=1}^{m_1}  \sum_{j=1}^{m_2} g(\sigma(e_i, \bar e_j), \tilde e_r)^2\notag\\
&+\sum_{r=1}^{m_2} \sum_{i, j=1}^{m_2} g(\sigma(\bar e_i, \bar e_j), \tilde e_r)^2++2\sum_{r=1}^{m_2} \sum_{i=1}^{m_2}  \sum_{j=1}^{m_3} g(\sigma(\bar e_i, e^*_j), \tilde e_r)^2\notag\\
&+\sum_{r=1}^{m_2} \sum_{i, j=1}^{m_3} g(\sigma(e^*_i, e^*_j), \tilde e_r)^2+2\sum_{r=1}^{m_2} \sum_{i=1}^{m_1}  \sum_{j=1}^{m_3} g(\sigma(e_i,  e^*_j), \tilde e_r)^2\notag\\
&+\sum_{r=m_2+1}^{m_2+m_3} \sum_{i, j=1}^{m_1} g(\sigma(e_i, e_j), \tilde e_r)^2+2\sum_{r=m_2+1}^{m_2+m_3} \sum_{i=1}^{m_1}  \sum_{j=1}^{m_2} g(\sigma(e_i, \bar e_j), \tilde e_r)^2\notag\\
&+\sum_{r=m_2+1}^{m_2+m_3}\sum_{i, j=1}^{m_2} g(\sigma(\bar e_i, \bar e_j), \tilde e_r)^2+2\sum_{r=m_2+1}^{m_2+m_3} \sum_{i=1}^{m_2}  \sum_{j=1}^{m_3} g(\sigma(\bar e_i, e^*_j), \tilde e_r)^2\notag\\
&+\sum_{r=m_2+1}^{m_2+m_3}\sum_{i, j=1}^{m_3} g(\sigma(e^*_i, e^*_j), \tilde e_r)^2+2\sum_{r=m_2+1}^{m_2+m_3} \sum_{i=1}^{m_1}  \sum_{j=1}^{m_3} g(\sigma(e_i,  e^*_j), \tilde e_r)^2.
\end{align}
We have no relation for warped product for the third, seventh, eleventh and twelfth terms, so leaving these terms. Then, using  Lemma \ref{wpl1} (ii) and Lemma \ref{wpl2} with the hypothesis of theorem, we derive
\begin{align}\label{wp4.3}
\|\sigma\|^2&\geq2\sum_{r=1}^{m_2} \sum_{i=1}^{p}  \sum_{j=1}^{m_2} g(\sigma(e_i, \bar e_j), \varphi \bar e_r)^2+2\sum_{r=1}^{m_2} \sum_{i=1}^{p}  \sum_{j=1}^{m_2} g(\sigma(\varphi e_i, \bar e_j), \varphi \bar e_r)^2\notag\\
&+2\sum_{r=1}^{m_2} \sum_{j=1}^{m_2} g(\sigma(e_{2p+1}, \bar e_j), \varphi \bar e_r)^2+\sum_{r=1}^{s}\sum_{i,j=1}^{m_2} g(\sigma(\bar e_i, \bar e_j), \csc\theta\,Fe^*_r)^2\notag\\
&+\sum_{r=1}^{s}\sum_{i,j=1}^{m_2} g(\sigma(\bar e_i, \bar e_j), \csc\theta\sec\theta\,FTe^*_r)^2.
\end{align}
Since, for a submanifold $M$ of a Sasakian manifold $\sigma(U, \xi)=-\varphi U$, for any $U\in\Gamma(TM)$, using this fact in the third term of \eqref{wp4.3}. Also, using Lemma \ref{wpl3} and Lemma \ref{wpl4} with the $\mathfrak{D}^\perp-\mathfrak{D}^\theta$ mixed totally geodesic condition, we derive
\begin{align}\label{wp4.4}
\|\sigma\|^2&\geq2\sum_{j, r=1}^{m_2} \sum_{i=1}^{p} \left(\varphi e_i(\ln f)+\eta(e_i)\right)^2g( \bar e_j, \bar e_r)^2+2\sum_{j, r=1}^{m_2} \sum_{i=1}^{p} \left(e_i(\ln f)\right)^2g(\bar e_j, \bar e_r)^2\notag\\
&+2\sum_{j,r=1}^{m_2}  g(\sigma(\varphi\bar e_j, \varphi \bar e_r)^2+\csc^2\theta\sum_{r=1}^{s}\sum_{i,j=1}^{m_2} \left(Te^*_r(\ln f)+\eta(e^*_r)\right)^2g(\bar e_i, \bar e_j)^2\notag\\
&+\cot^2\theta\sum_{r=1}^{s}\sum_{i,j=1}^{m_2} \left(e^*_r(\ln f)\right)^2g(\bar e_i, \bar e_j)^2.
\end{align}
Now, we consider both cases: (i) When $\xi\in\Gamma(\mathfrak{D})$, then we have
\begin{align*}
\|\sigma\|^2&\geq2m_2\sum_{i=1}^{2p+1} \left(e_i(\ln f)\right)^2-2m_2\left(e_{2p+1}(\ln f)\right)^2+2m_2\notag\\
&+m_2\csc^2\theta\sum_{r=1}^{m_3}\left(Te^*_r(\ln f)\right)^2+m_2\cot^2\theta\sum_{r=1}^{s}\left(e^*_r(\ln f)\right)^2\notag\\
&-m_2\csc^2\theta\sum_{r=s+1}^{m_3}\left(Te^*_r(\ln f)\right)^2.
\end{align*}
Now, using gradient definition and Lemma \ref{wpl1} (i), we find
\begin{align*}
\|\sigma\|^2&\geq2m_2\left(\|\nabla^T(\ln f)\|^2+1\right)+m_2\csc^2\theta\|T\nabla^\theta(\ln f)\|^2\notag\\
&+m_2\cot^2\theta\sum_{r=1}^{s}\left(e^*_r(\ln f)\right)^2-m_2\csc^2\theta\sec^2\theta\sum_{r=1}^{s}g(Te^*_r, T\nabla^\theta(\ln f))^2\notag\\
&=2m_2\left(\|\nabla^T(\ln f)\|^2+1\right)+m_2\csc^2\theta\|\nabla^\theta(\ln f)\|^2,
\end{align*}
which is inequality (i). If $\xi\in\Gamma(\mathfrak{D}^\theta)$, then from \eqref{wp4.3}, we obtain

\begin{align*}
\|\sigma\|^2&\geq2m_2\|\nabla^T(\ln f)\|^2+m_2\csc^2\theta\sum_{r=1}^{m_3} g\left(e^*_r, T\nabla^\theta(\ln f)\right)^2+m_2\csc^2\theta\notag\\
&+m_2\cot^2\theta\sum_{r=1}^{s}\left(e^*_r(\ln f)\right)^2-m_2\csc^2\theta\sum_{r=1}^{s}g(e^*_{r+s}, T\nabla^\theta(\ln f))^2-m_2\csc^2\theta\notag\\
&=2m_2\|\nabla^T(\ln f)\|^2+m_2\csc^2\theta\|T\nabla^\theta(\ln f)||^2\notag\\
&+m_2\cot^2\theta\sum_{r=1}^{s}\left(e^*_r(\ln f)\right)^2-m_2\csc^2\theta\sec^2\theta\sum_{r=1}^{s}g(Te^*_{r}, T\nabla^\theta(\ln f))^2\notag\\
&=2m_2\|\nabla^T(\ln f)\|^2+m_2\cot^2\theta\|\nabla^\theta(\ln f)||^2,
\end{align*}
which is inequality (ii). For the equality case, From the leaving and vanishing terms in \eqref{wp4.1} and  \eqref{wp4.2}, we obtain
\begin{align}\label{wp4.5}
\sigma(\mathfrak{D}, \mathfrak{D})=0,\,\,\sigma(\mathfrak{D}^\perp, \mathfrak{D}^\theta)=0,\,\,\sigma(\mathfrak{D}^\theta, \mathfrak{D}^\theta)=0,\,\,\,\sigma(\mathfrak{D}, \mathfrak{D}^\theta)=0.
\end{align}
Then, from \eqref{wp4.5} with the Remark \ref{wpr1}, we conclude that $B$ is totally geodesic in $\tilde M$. Also, we find
\begin{align}\label{wp4.6}
\sigma(\mathfrak{D}, \mathfrak{D}^\perp)\subseteq\varphi\mathfrak{D}^\perp,\,\,\,\sigma(\mathfrak{D}^\perp, \mathfrak{D}^\perp)\subseteq F\mathfrak{D}^\theta.
\end{align}
Thus, by Remark \ref{wpr1} with \eqref{wp4.5} and \eqref{wp4.6}, we deduce that $M_\perp$ is totally umbilical in $\tilde M$. Hence, the theorem is proved completely.
\end{proof}
\section{Special cases of Theorem \ref{wpt3}}
There are two well known special cases of Theorem \ref{wpt3} given below:\\

\noindent1. If $\mathfrak{D}^\theta=\{0\}$ i.e., $\dim M_\theta=0$ in a contact skew CR-warped product, then it reduces to contact CR-warped products of the form $M=M_T\times_fM_\perp$ studied in \cite{Hasegawa}. In this case, the statement of Theorem \ref{wpt3} will be:
{\textit{Let $M=M_T\times_fM_\perp$ be a contact CR-warped product submanifold of a Sasakian manifold $\tilde M$ such that $\xi$ is tangent to $M_T$, where $M_T$ and $M_\perp$ are invariant and anti-invariant submanifolds  of $\tilde M$ with their real dimensions $m_1,\, m_2$, respectively. Then we have:
\begin{enumerate}
\item [(i)] The squared norm of the second fundamental from $\sigma$ satisfies 
\begin{align*}
\|\sigma\|^2\geq2m_2\left(\|\nabla^T(\ln f)\|^2+1\right).
\end{align*}
where $\nabla^T(\ln f)$ is the gradient of $\ln f$ along $M_T$.
\item [(ii)] If the equality sign holds in above inequality, then $M_T$ is totally geodesic and $M_\perp$ is a totally umbilical in $\tilde M$.
\end{enumerate}}}
\noindent Which is the main result of \cite{Hasegawa}.\\

\noindent2. On the other hand, if $\mathfrak{D}=\{0\}$ in a contact skew CR-warped product, then it will change into a pseudo-slant warped product of the form $M=M_\theta\times_fM_\perp$ studied in \cite{SU3}. In this case,Theorem 4.2 of \cite{SU3} is a particular case of Theorem \ref{wpt3} as follows:
\begin{corollary} (\textit{ Theorem 4.2 of \cite{SU3}}) Let $M=M_\theta\times_fM_\perp$ be a mixed totally geodesic warped product submanifold of a Sasakian manifold $\widetilde M$ such that $\xi\in\Gamma(\mathfrak{D}^\theta)$, where $M_\theta$ is a proper slant submanifold and $M_\perp$ is an $m_2$-dimensional anti-invariant submanifold of $\widetilde M$. Then we have:
\begin{enumerate}
\item [(i)] The squared norm of the second fundamental form of $M$ satisfies
\begin{align*}
\|\sigma\|^2\geq m_2\cot^2\theta\,\|\nabla^\theta(\ln f)\|^2
\end{align*}
where $\nabla^\theta\ln f$ is the gradient of $\ln f$ along $M_\theta$.
\item [(ii)] If the equality sign in (i) holds identically, then $M_\theta$ is totally geodesic in $\widetilde M$ and $M_\perp$ is a totally umbilical submanifold of $\widetilde M$.
\end{enumerate}
\end{corollary}

\section{Examples}
We construct the following non-trivial examples of Riemannian products and contact skew CR-warped products in Euclidean spaces.
\begin{example}
\label{ex3}
\rm{Let $M$ be a submanifold of Euclidean $9$-space ${\mathbb{R}}^{9}$ with the cartesian coordinates $(x_1,\,\cdots, x_4,\, y_1,\,\cdots, y_4\,, t)$ and the almost contact structure defined in Example \ref{ex1}.
If $M$ is given by the equations 
\begin{align*}
&x_1=u_1\,\,y_1=v_1,x_2=u_2,\,\,y_2=v_2,\,\,x_3=\sin v_2,\,\,y_3=\cos v_2,\\
&x_4=\cos w^2,\,y_4=\sin w^2,\,\,t=t,
\end{align*}
then, the tangent space $TM$ is spanned by $X_1,\,X_2,\,X_3,\,X_4,\,X_5$ and $X_6$, where
\begin{align*}
&X_1=\frac{\partial}{\partial x_1},\,\,X_2=\frac{\partial}{\partial y_1},\,\,X_3=\frac{\partial}{\partial x_2},\,\,X_4=\cos v_2\,\frac{\partial}{\partial x_3}+\frac{\partial}{\partial y_2}-\sin v_2\,\frac{\partial}{\partial y_3},\\
&X_5=-2w\sin w^2\,\frac{\partial}{\partial x_4}+2w\cos w^2\frac{\partial}{\partial y_4},\,\,X_6=\frac{\partial}{\partial t}.
\end{align*}
Then, we find that ${\mathfrak{D}}=\rm{Span}\{X_1, X_2\}$ is an invariant distribution and ${\mathfrak{D}}^\perp=\rm{Span}\{X_5\}$ is an anti-invariant distribution. Moreover, ${\mathfrak{D}}^{\theta}=\rm{Span}\{X_3, X_4\}$ is a slant distribution with slant angle $\theta=45^{\circ}$. Hence, $M$ is a skew CR-submanifold of ${\mathbb{R}}^{9}$. Clearly, each distribution is integrable. If $M_T,\,M_\theta$ and $M_\perp$ integral manifolds of ${\mathfrak{D}},\,{\mathfrak{D}}^\theta$ and ${\mathfrak{D}^{\perp}}$, respectively, then $M$ is a Riemannian product submanifold of $B=M_T\times M_\theta$ and $M_\perp$ in ${\mathbb{R}}^{9}$.}
\end{example}

\begin{example}
\label{ex4}
\rm{Consider the Euclidean space ${\mathbb{R}}^{13}$ with the cartesian coordinates $(x_1,\,\cdots, x_6,\, y_1,\,\cdots, y_6\,, z)$ and the almost contact structure
\begin{align*}
\varphi\left(\frac{\partial}{\partial x_i}\right)=-\frac{\partial}{\partial y_i},~~~~\varphi\left(\frac{\partial}{\partial y_j}\right)=\frac{\partial}{\partial x_j},~~~\varphi\left(\frac{\partial}{\partial z}\right)=0,~~~1\leq i, j\leq6.
\end{align*}
It is clear that $\mathbb R^{13}$ is an almost contact metric manifold with respect to the given structure and standard Euclidean metric tensor of $\mathbb R^{13}$. Let  $M$ be a submanifold of ${\mathbb{R}}^{13}$ defined by the immersion $\psi:{\mathbb{R}}^{7}\to {\mathbb{R}}^{13}$ as follows
\begin{align*}
&\psi(u, v, w, r, s, t, z)=(u\cos(w+r), u\sin(w+r), v\cos(w-r), v\sin(w-r), k(u+v),\\
& s+t, v\cos(w+r), v\sin(w+r), u\cos(w-r), u\sin(w-r), -k(u-v), -s+t, z)
\end{align*}
for non-zero vectors and a scalar $k\neq0$. Let the tangent space of $M$ is spanned by the following vectors
\begin{align*}
&X_1=\cos(w+r)\,\frac{\partial}{\partial x_1}+\sin(w+r)\,\frac{\partial}{\partial x_2}+k\frac{\partial}{\partial x_5}+\cos(w-r)\,\frac{\partial}{\partial y_3}\\
&\hspace{5mm}+\sin(w-r)\,\frac{\partial}{\partial y_4}-k\frac{\partial}{\partial y_5},\\
&X_2=\cos(w-r)\,\frac{\partial}{\partial x_3}+\sin(w-r)\,\frac{\partial}{\partial x_4}+k\frac{\partial}{\partial x_5}+\cos(w+r)\,\frac{\partial}{\partial y_1}\\
&\hspace{5mm}+\sin(w+r)\,\frac{\partial}{\partial y_2}+k\frac{\partial}{\partial y_5},\\
&X_3=-u\sin(w+r)\,\frac{\partial}{\partial x_1}+u\cos(w+r)\,\frac{\partial}{\partial x_2}-v\sin(w-r)\,\frac{\partial}{\partial x_3}+v\cos(w-r)\,\frac{\partial}{\partial x_4}\\
&\hspace{5mm}-v\sin(w+r)\,\frac{\partial}{\partial y_1}+v\cos(w+r)\,\frac{\partial}{\partial y_2}-u\sin(w-r)\,\frac{\partial}{\partial y_3}+u\cos(w-r)\,\frac{\partial}{\partial y_4},\\
&X_4=-u\sin(w+r)\,\frac{\partial}{\partial x_1}+u\cos(w+r)\,\frac{\partial}{\partial x_2}+v\sin(w-r)\,\frac{\partial}{\partial x_3}-v\cos(w-r)\,\frac{\partial}{\partial x_4}\\
&\hspace{5mm}-v\sin(w+r)\,\frac{\partial}{\partial y_1}+v\cos(w+r)\,\frac{\partial}{\partial y_2}+u\sin(w-r)\,\frac{\partial}{\partial y_3}-u\cos(w-r)\,\frac{\partial}{\partial y_4},\\
&X_5=\frac{\partial}{\partial x_6}-\frac{\partial}{\partial y_6},\,\,X_6=\frac{\partial}{\partial x_6}+\frac{\partial}{\partial y_6},\,\,X_7=\frac{\partial}{\partial z}.
\end{align*}
Then, the distribution ${\mathfrak{D}}^\perp=\rm{Span}\{X_3, X_4\}$  is an anti-invariant distribution. It is easy to see that ${\mathfrak{D}}=\rm{Span}\{X_5, X_6\}$ is an invariant distribution and ${\mathfrak{D}}^{\theta}=\rm{Span}\{X_1, X_2\}$ is a slant distribution with slant angle $\theta=\cos^{-1}\left(\frac{k^2}{1+k^2}\right)$. Hence, $M$ is a proper skew CR-submanifold of order $1$ of ${\mathbb{R}}^{13}$ such that $\xi=\frac{\partial}{\partial z}$ is tangent to $M$. It is easy to observe that each distribution is integrable.  If we denote the integral manifolds of ${\mathfrak{D}}$, ${\mathfrak{D}}^\theta$ and ${\mathfrak{D}^{\perp}}$ by $M_T$, $M_\theta$ and $M_\perp$, respectively, then the induced metric tensor $g$ of $M$ is given by
\begin{align*}
g&=2(1+k^2)(du^2+dv^2)+2(ds^2+dt^2)+dz^2+2(u^2+v^2)(dw^2+dr^2)\notag\\
&=g_{B}+2(u^2+v^2)g_{M_{\perp}}.\notag
\end{align*}
Hence, $M$ is a skew CR-warped product submanifold of ${\mathbb{R}}^{13}$ with the warping function $f=\sqrt{2(u^2+v^2)}$ and the warped product metric $g$ such that $(B, g_1)=(M_T\times M_\theta, g_1)$ with product metric $g_1=2(1+k^2)(du^2+dv^2)+2(ds^2+dt^2)+dz^2$.}
\end{example}

\begin{acknowledgements}
This project was funded by the Deanship of Scientific Research (DSR), King Abdulaziz University, Jeddah, Saudi Arabia under grant no. (KEP-PhD-88-130-38). The authors, therefore, acknowledge with thanks DSR technical and financial support.
\end{acknowledgements}


\begin{thebibliography}{99}
\bibitem{Bejancu} A. Bejancu, \textit{Geometry of CR-submanifolds}, Kluwer Academic Publishers, Dordrecht, 1986.
\bibitem{Bishop} R. L. Bishop and B. O'Neill, { \textit Manifolds of negative curvature,} Trans. Amer. Math. Soc. {\bf{145}} (1969), 1-49.
\bibitem{Blair} D. E. Blair, { \textit Contact manifolds in Riemannian geometry,} Lecture Notes in Math. {\bf{509}}, Springer, Berlin, 1976.
\bibitem{Cab1} J. L. Cabrerizo, A. Carriazo, L.M. Fernandez and M. Fernandez, { \textit Semi-slant submanifolds of a Sasakian manifold}, Geom. Dedicata {\bf{78}} (1999), 183-199.
\bibitem{Cab2} J. L. Cabrerizo, A. Carriazo, L.M. Fernandez and M. Fernandez, { \textit Slant submanifolds in Sasakian manifolds}, Glasgow Math. J. {\bf{42}} (2000), 125-138.
\bibitem{Carriazo} A. Carriazo, { \textit{New developments in slant submanifolds theory,}} Narosa Publishing House, New Delhi, 2002.
\bibitem{C1} B.-Y. Chen, { \textit Slant immersions}, Bull. Austral. Math. Soc. {\bf{41}} (1990), 135-147.
\bibitem{C2} B.-Y. Chen, { \textit Geometry of slant submanifolds,} Katholieke Universiteit Leuven,  1990.
\bibitem{C3} B.-Y. Chen, { \textit Geometry of warped product CR-submanifolds in Kaehler manifolds}, Monatsh. Math. {\bf{133}} (2001), 177-195.
\bibitem{C4} B.-Y. Chen, { \textit Geometry of warped product CR-submanifolds in Kaehler manifolds II}, Monatsh. Math. {\bf{134}} (2001), 103-119.
\bibitem{C6} B.-Y. Chen, { \textit  Another general inequality for CR-warped products in complex space forms}, Hokkaido Math. J. {\bf{32}} (2003), no. 2, 415-444.
\bibitem{C7} B.-Y. Chen, { \textit  CR-warped products in complex projective spaces with compact holomorphic factor}, Monatsh. Math. {\bf{141}} (2004), no. 3, 177-186.
\bibitem{book} B.-Y. Chen, { \textit Pseudo-Riemannian geometry, $\delta$-invariants and applications},  World Scientific, Hackensack, NJ, 2011.
\bibitem{C5} B.-Y. Chen, { \textit Geometry of warped product submanifolds: a survey}, J. Adv. Math. Stud. {\bf{6}} (2013), no. 2, 1--43.
\bibitem{book17} B.-Y. Chen, { \textit Differential geometry of warped product manifolds and submanifolds},  World Scientific, Hackensack, NJ, 2017.
\bibitem{CU} B.-Y. Chen and S. Uddin, \textit{Warped Product Pointwise Bi-slant Submanifolds of Kaehler Manifolds}, Publ. Math. Debrecen {\bf{92}} (2018), no. 1-2, 183-199.
\bibitem{Haider} S.M. Khursheed Haider, M. Thakur and Advin, \textit{Warped product skew CR-submanifolds of cosymplectic manifolds}, Lobachevskii J. Math. {\bf{33}} (2012), no. 3, 262-273.
\bibitem{Hasegawa} I. Hasegawa and I. Mihai, { \textit Contact CR-warped product submanifolds in Sasakian manifolds}, Geom. Dedicata {\bf{102}} (2003), 143-150.
\bibitem{Lotta} A. Lotta, { \textit Slant submanifolds in contact geometry}, Bull. Math. Soc. Roumanie {\bf{39}} (1996), 183-198.
\bibitem{Monia} M.F. Naghi, I. Mihai, S. Uddin and F. R. Al-Solamy, \textit{Warped product skew CR-submanifolds of Kenmotsu manifolds and their applications}, Filomat {\bf{32}} (2018), no. 10, 1-24.
\bibitem{Ronsse} G. S. Ronsse, \textit{Generic and skew CR-submanifolds of a Kaehler manifold}, Bull. Inst. Math. Acad. Sinica {\bf{18}} (1990), 127-141.
\bibitem{Sahin1} B. Sahin, { \textit Warped product submanifolds of Kaehler manifolds with a slant factor}, Ann. Pol. Math. {\bf{95}} (2009), 207-226.
\bibitem{Sahin2} B. Sahin, \textit{Skew CR-warped products of Kaehler manifolds}, Math. commun. {\bf{15}} (2010), 189-204.
\bibitem{Tripathi} M.M. Tripathi, \textit{Almost semi-invariant submanifolds of trans-Sasakian manifolds,} J. Indian Math. Soc. (N.S.)  {\bf{62}} (1996), no. 1-4, 225-245.


\bibitem{UC11} S.Uddin, A. Y. M. Chi,  {\textit{Warped product pseudo-slant submanifolds of nearly Kaehler manifolds}} An. Stiint. Univ. ``Ovidius'' Constanta Ser. Mat. {\bf 19} (2011), no. 3, 195--204. 

\bibitem{ssdd44} S. Uddin, A. Mustafa,  B. R. Wong, C. Ozel, {\textit{A geometric inequality for warped product semi-slant submanifolds of nearly cosymplectic manifolds}},  Rev. Un. Mat. Argentina {\bf{55}} (2014), 55-69.

\bibitem{SU2} S. Uddin, F. R. Al-Solamy, {\textit{Warped product pseudo-slant submanifolds of cosymplectic manifolds}}, An. \c{S}tiin\c{t}. Univ. Al. I. Cuza Ia\c{s}i Mat. (N.S.), 62 (2016), 901-913.

\bibitem{UAK16} S. Uddin, F. R. Al-Solamy, K. A. Khan, {\textit{Geometry of warped product pseudo-slant submanifolds in nearly Kaehler manifolds}}, An. Stiint. Univ. Al. I. Cuza Iasi. Mat. (N.S.) {\bf 62}, vol. 3, (2016), no. 2, 927--938.

\bibitem{SC} S. Uddin,  B.-Y. Chen, F. R. Al-Solamy, {\textit{Warped product bi-slant immersions in Kaehler manifolds}} Mediterr. J. Math. {\bf 14} (2017),  Art. no. 95, 11 pp .

\bibitem{SU3} S. Uddin and F. R. Al-Solamy, \textit{Warped product pseudo-slant immersions in Sasakian manifolds}, Publ. Math. Debrecen, {\bf{91}} (2017), no. 3-4, 331--348, doi: 10.5486/PMD.2017.7640.

\bibitem{UAAM20}  S. Uddin,  L. S. Alqahtani,   A. H. Alkhaldi, F. Y. Mofarreh, {\textit{CR-slant warped product submanifolds in nearly Kaehler manifolds}} Int. J. Geom. Methods Mod. Phys. {\bf 17} (2020), no. 1, Art. no. 2050003, 11 pp.

\bibitem{UCAA20} S. Uddin, B.-Y. Chen,  A. AL-Jedani,  A., Alghanemi,   {\textit{Bi-warped product submanifolds of nearly Kaehler manifolds}} Bull. Malays. Math. Sci. Soc. {\bf 43}   (2020), no. 2, 1945--1958.

\bibitem{Yano} K. Yano and M. Kon, {\textit{CR submanifolds of Kaehlerian and Sasakian manifolds}}, Progress in Mathematics 30, Birkhauser, Boston, Mass., 1983.
\end{thebibliography}


\end{document}